\documentclass[3p,sort&compress]{personalartcl}
\usepackage[T1]{fontenc}

\usepackage{enumerate}
\usepackage{enumitem}
\usepackage[british]{babel}
\usepackage[utf8]{inputenc}
\usepackage{dsfont,bbm,mathrsfs}
\usepackage{amsmath, amsthm, amsbsy, amssymb}
\usepackage[all]{xy}
\usepackage{tikz}
\usetikzlibrary{matrix,arrows,chains,positioning,scopes,decorations.pathmorphing,backgrounds,fit}

\renewcommand{\leq}{\leqslant}
\renewcommand{\geq}{\geqslant}

\DeclareMathOperator{\R}{\mathbb{R}}

\newcommand{\Log}{{\mathscr{L}}} % a real-valued logic
\newcommand{\Mog}{\mathscr{M}} % a many-valued logic
\newcommand{\provesL}{{\vdash_{\Log}}} % its consequence relation
\newcommand{\nprovesL}{{{\not \vdash}_{\Log}}} % its consequence relation
\newcommand{\BL}{{\sf BL}} %Basic logic
\newcommand{\G}{{\sf G}} % Goedel logic
\newcommand{\LL}{{\sf L}} % Luk logic
\newcommand{\PP}{{\sf P}} % Product logic

\newcommand{\Var}{\textup{{\sc Var}}}
\renewcommand{\Form}{\textup{{\sc Form}}}

\newcommand{\ie}{\textit{i.e.}}
\newcommand{\cf}{\textit{cf.}}
\newcommand{\eg}{\textit{e.g.}}

\newcommand{\princ}[2]{\\[6pt] {\bf #1.}\,\,{#2}}
\renewcommand{\P}[1]{{\textup{P#1}}}

\theoremstyle{theorem}
\newtheorem*{theoremI}{Theorem I}
\newtheorem*{theoremII}{Theorem II}
\newtheorem*{theoremIII}{Theorem III}

\newtheorem{theorem}{Theorem}[section]

\newtheorem{lemma}[theorem]{Lemma}
\newtheorem{corollary}[theorem]{Corollary}
\newtheorem*{corollarynonum}{Corollary}

\newtheorem{claim}[theorem]{Claim}

\theoremstyle{definition}

\newtheorem{remark}[theorem]{Remark}
\begin{document}
\begin{frontmatter}
\title{Two principles in many-valued logic\\[12pt]{\normalsize \sl For Petr H\'ajek}}
\author{Stefano Aguzzoli}%\corref{cor}}
\ead{aguzzoli@di.unimi.it}
\address{Dipartimento di Informatica, Universit\`a degli
Studi di Milano. Via Comelico 39--41, 20135 Milano, Italy}
\author{{Vincenzo Marra}\corref{cor}}
\ead{vincenzo.marra@unimi.it}
\address{Dipartimento di Matematica ``Federigo Enriques'', Universit\`a degli
Studi di Milano. Via Cesare Saldini 50, 20133 Milano, Italy}
\journal{H\'ajek's Festschrift. Revised version submitted}

\cortext[cor]{Corresponding author.}
%%% Keywords
\begin{abstract}Classically, two propositions are logically equivalent precisely when they are true under the same logical valuations. Also, two
logical valuations are distinct if, and only if, there is a formula that is true according to one valuation, and false according to the other. By a \emph{real-valued logic} we mean a many-valued logic in the sense of Petr H\'ajek that is complete with respect to a subalgebra of truth values of a BL-algebra given by a continuous triangular norm on $[0,1]$.  Abstracting the two foregoing properties from classical logic leads us to two principles that a real-valued logic  may or may not satisfy. We prove that the two principles are sufficient to characterise {\L}ukasiewicz and  G\"odel logic,  to within extensions. We also prove that, under the additional assumption that the set of truth values be closed in the Euclidean topology of $[0,1]$, the two principles also afford a characterisation of Product logic.
\end{abstract}
\end{frontmatter}

\section{Prologue.}\label{s:pro}
At the outset of his landmark monograph \cite{hajek}, Petr H\'ajek writes:
\begin{quote}{\sl There are various systems of fuzzy logics, not just one. We have one basic logic \textup{(BL)} and
three of its most important extensions: {\L}ukasiewicz logic, G\"odel logic, and the product logic.} \cite[p.5]{hajek}.
\end{quote}
\emph{Basic Logic}  is, of course, the creation of H\'ajek himself. One of its several virtues is to afford metamathematical
comparison of  many-valued  logics to an unprecedented degree of clarity. Our paper is intended as a  modest contribution to such comparative studies; it will soon transpire that it would have been impossible to
write it, in the possible but unfortunate worlds orphan of \cite{hajek}.

\smallskip We assume familiarity with Basic (propositional) Logic, triangular norms (\emph{t-norms}, for short), and BL-algebras; see \cite{hajek}, and Section \ref{s:pre} for an outline. Note that in this paper `t-norm' means  `continuous t-norm', for the sake of brevity. We write
 $\Form$ for the set of formul\ae\  over the countable collection of propositional variables $\Var:=\{X_1,X_2,\ldots\}$, with primitive connectives $\to$ (implication), $\&$ (monoidal conjunction), and $\bot$ (\textit{falsum}). As usual, $\&$ is semantically interpreted by a t-norm, $\to$ by its residuum, and $\bot$ by $0$. We adopt the standard abbreviations, $\neg \alpha:=\alpha \to \bot$, $\alpha \wedge \beta:=\alpha \& (\alpha\to\beta) $, $\alpha \vee \beta:=((\alpha\to\beta)\to \beta)\wedge((\beta \to \alpha)\to \alpha)$,  and $\alpha \leftrightarrow \beta := (\alpha \to \beta) \& (\beta \to \alpha)$. We write \BL\ to denote Basic Logic, as axiomatised  in \cite{hajek, cignolietal_bl}. An \emph{extension} of \BL\ is a   collection of formul\ae\ closed under the (syntactic) consequence relation of \BL, and closed under substitutions. If $\Mog$ is an extension of \BL, we always tacitly assume that $\Mog$ is consistent,   we refer to $\Mog$ as a \emph{many-valued logic}, and we denote by ${\vdash_{\Mog}}$  its consequence relation.

\emph{{\L}ukasiewicz logic}, denoted \LL, is obtained by extending \BL\ with the axiom schema $\neg\neg\varphi \to \varphi$. \emph{G\"odel logic}, denoted \G, is obtained by adding to \BL\ the schema $\varphi \to (\varphi\&\varphi)$. To obtain \emph{Product logic}, written \PP, one extends \BL\ with $\neg \varphi \vee ((\varphi\to(\varphi \&\psi))\to \psi)$. See \cite[p.63, Definitions 4.2.1 and 4.1.1]{hajek}, and \cite[Chapter I]{handbook1}.

Over the real unit interval $[0,1]\subseteq \R$, consider  a BL-algebra $([0,1],*,\to_*,0)$, where  $*\colon [0,1]\times [0,1]\to[0,1]$ is a continuous t-norm  with  residuum $\to_{*}$. By an \emph{algebra of truth values} we shall mean a subalgebra $T_{*}$ of some such BL-algebra $([0,1],*,\to_*,0)$. Note, in particular, that $\{0,1\}$ is a subset of any algebra of truth values. We write $T_{*}\subseteq [0,1]$ for the underlying set of the algebra of truth values, too, \ie\ for the set of truth values itself.

 We say that the pair $(\Log, T_*)$ is a \emph{real-valued logic} if $\Log$ is an extension of \BL\ that   is complete with respect to valuations $\mu\colon \Form \to T_{*}$ into the given algebra of truth values. Any algebra of truth values $T'_*$ such that $(\Log,T'_*)$ is a real-valued logic is said to \emph{induce} $\Log$. When $T_{*}=[0,1]$, we also say that $\Log$ is induced by the t-norm $*$. (This   makes sense, recalling that $\to_{*}$ is uniquely determined by $*$. See Section \ref{s:pre} below.)  Distinct algebras of truth values may of course induce the same  logic $\Log$, \ie\ the same extension of \BL. When we write that $\Log$ is a real-valued logic, with no reference to $T_*$, we mean that there is at least one algebra of truth values $T_*$ that induces $\Log$. 
 
\smallskip With this machinery in place, we consider two   principles that a real-valued logic $\Log$ may or may not satisfy.
\smallskip
\princ{\P{1}}{For every algebra $T_*$ of truth values inducing $\Log$, the following holds. For each  $\alpha, \beta \in\Form$,  we have $\vdash_{\Log} \alpha \leftrightarrow \beta$ if, and only if,
\begin{align*}
\mu(\alpha) = 1 \quad  \Longleftrightarrow  \quad \mu(\beta) = 1
\end{align*}
holds for each valuation $\mu \colon \Form \to T_{*}$.\qed\\}
\princ{\P{2}}{For every algebra  $T_*$ of truth values inducing $\Log$, the following holds. For each pair of valuations $\mu, \nu \colon \Form \to T_{*}$, if $\mu\neq \nu$ then
there is a formula $\alpha\in \Form$ such that $\mu(\alpha) >0$ while $\nu(\alpha) = 0$.\qed\\}

\noindent Our first two results are that \P{1} and \P{2} are characteristic of $\G$ and $\LL$, respectively, to within extensions.
\begin{theoremI}A real-valued logic $(\Log,T_*)$ satisfies $\P{1}$ if, and only if, $\Log$ is an extension of G\"odel logic.
\end{theoremI}
\begin{theoremII}A real-valued logic $(\Log,T_*)$ satisfies $\P{2}$ if, and only if, $\Log$ is an  extension of {\L}ukasiewicz logic.
\end{theoremII}
\begin{remark}
Observe that the two preceding theorems show that in \P{1} and \P{2} one can safely replace the initial universal quantification by an existential one. In other words, the principles \P{1} and \P{2} display robustness with respect to the specific choice of the algebra of truth values, \textit{salva logica} $\Log$.\qed
\end{remark}

\noindent We prove Theorem I in Section \ref{s:tarski}, and Theorem II in Section \ref{s:leibniz}, after some preliminaries in Section \ref{s:pre}.

\smallskip The question arises, can one also characterise Product logic by means of  general principles such as \P{1} and \P{2}. We shall show how to answer this question affirmatively, under one additional assumption. Let us say that the real-valued logic $\Log$ is \emph{closed} if there exists an algebra of truth values $T^*$ inducing $\Log$ such that the underlying set of $T^*$ is closed in the Euclidean topology of $[0,1]$.
Product logic is the unique \emph{closed} real-valued logic that \emph{fails} both \P{1} and \P{2} hereditarily with respect to real-valued extensions, in the following sense:
\begin{theoremIII}A closed real-valued logic $\Log$ is Product logic if, and only if, $\Log$ and all of its non-classical, real-valued extensions fail \P{1} and \P{2}.
\end{theoremIII}

\noindent We prove Theorem III in Section \ref{s:product}.

\smallskip The proofs of Theorems I--III are relatively straightforward applications of known facts about extensions of Basic Logic.
The interest of the present contribution, if any, is thus to be sought not so much in the technical depth of the results, as in the significance of the two principles \P{1} and \P{2} in connection with logics of comparative truth. Before turning to the proofs, let us therefore expound  on \P{1} and \P{2} a little.

Logics  fulfilling \P{1} share with classical logic the feature that each proposition is uniquely determined, up to logical equivalence, by the collection of its true interpretations (that is, models), where `true' in the latter sentence is to be read as `true to degree $1$'. In the classical case this may be conceived as a consequence of the Principle of Bivalence, along with completeness. (Indeed, if in classical logic $\alpha$ and $\beta$ evaluate to $1$ at exactly the same $\mu$'s, then, by bivalence, they evaluate to
 the same value at each $\mu$; hence $\alpha \leftrightarrow\beta$ is a tautology, and we therefore have $\vdash \alpha \leftrightarrow \beta$  by completeness.) Theorem I shows that, remarkably, real-valued logics that \emph{fail} the Principle of Bivalence---for instance, G\"{o}del logic---may still satisfy \P{1}.

Logics fulfilling \P{2} share with classical logic the feature that distinct models of the logic can be separated by some formula. In more detail, classical logic has the property that if $\mu$ and $\nu$ are two distinct true interpretations  of its axioms, then there is a formula $\alpha$ that can tell apart the two models $\mu$ and $\nu$, in the sense that $\alpha$ is not false (\ie\ true) in $\mu$ but false in $\nu$.  A logic failing \P{2}, by contrast, must allow two distinct true interpretations $\mu$ and $\nu$ of its axioms which are indiscernible, in the sense that no proposition  is false (\ie\ evaluates to  degree $0$) in $\nu$ and  not false (\ie\ evaluates to degree $>0$) in $\mu$.\footnote{It should be emphasised that there is some leeway in formulating the separating conditions $\mu(\alpha)>0$ and $\nu(\alpha)=0$ here: see Corollary \ref{cor:equiv} below for  equivalent variants.} In this precise sense, the given real-valued semantics of such a logic is redundant, as one could identify $\mu$ and $\nu$ without any logical loss. Theorem II shows that, remarkably, there is just one $[0,1]$-valued logic---namely, {\L}ukasiewicz logic---capable of avoiding that redundancy, by actually telling apart any two distinct real numbers in $[0,1]$.
\section{Preliminary facts about real-valued logics.}\label{s:pre}
We outline the framework of H\'ajek's Basic Logic.
A (\emph{continuous}) {\em t-norm} is a binary operation $* \colon [0,1]^2 \to [0,1]$, continuous with respect to the Euclidean topology of $[0,1]$, that is associative, commutative, has $1$ as neutral element, and is monotonically non-decreasing in each argument:
\begin{align*}
\forall a,b,c \in [0,1]\ : \ b \leq c\ \Longrightarrow \ a * b \leq a * c.
\end{align*}
For $a,b \in [0,1]$, set $a \to_* b := \sup{\{ c \in [0,1] \mid a * c \leq b \}}$.
It is well known \cite[Sec.\ 2.1.3]{hajek} that continuity is sufficient to entail $a \to_* b = \max{\{ c \in [0,1] \mid a * c \leq b \}}$. The operation $\to_*$ is called the {\em residuum} of $*$.
Recall that the residuum determines the underlying order, that is, $a \leq b$ if, and only if, $a \to_* b = 1$. Further recall that the subset of $\Form$ that evaluates to $1$ under every valuation $\mu\colon \Form \to ([0,1],*,\to_{*},\bot)$, is by definition the collection of all tautologies of \BL. It is one of the main achievements of \cite{hajek}, of course, that this set is recursively axiomatisable by schemata, using \textit{modus ponens} as the only deduction rule; see also \cite{cignolietal_bl} for an improved axiomatisation. Moreover, \BL\ is an algebraizable logic, see \cite[p.25 and references therein]{hajek}; the algebras in the corresponding variety are called \emph{BL-algebras}, and schematic extensions of \BL\ are in one-one natural correspondence with subvarieties of BL-algebras. Each t-norm $*\colon [0,1]^{2} \to [0,1]$ induces a BL-algebra $([0,1],*,\to_{*},0)$, and the variety of BL-algebras is generated by the collection of all t-norms. More generally, each algebra of truth values as defined above is a BL-algebra. We occasionally write `BL-chain' for `totally ordered BL-algebra'.

Given algebras of truth-values $T_*, T_{*}' \subseteq [0,1]$, we say that $\sigma \colon T_* \to T_{*}'$ is
an {\em isomorphism} if $\sigma$ is an isomorphism of BL-algebras; equivalently, $\sigma$ is a bijection,
for all $a,b \in T_*$ we have $\sigma( a * b ) = \sigma(a) * \sigma(b)$, and
$a \leq b$ implies $\sigma(a) \leq \sigma(b)$.

Recall the  three fundamental t-norms.
\begin{align}
x \odot y &:= \max\{0,x+y-1\}\\
x \min y&:=\min\{x,y\}\label{e:gtnorm}\\ 
x \times y &:= xy \label{e:prodtnorm}
\end{align}
The associated residua evaluate to $1$ for each $x,y \in [0,1]$ with $x \leq y$; when $x >y$, they are respectively given by:
\begin{align*}
x \to_\odot y &:= \min\{1,1-x+y\} \\ 
x \to_{\min} y &:=y \\ 
x \to_\times y &:= \frac{y}{x}  
\end{align*}
The algebra of truth values $T_\odot:=([0,1],\odot,\to_{\odot},0)$ is called the \emph{standard MV-algebra};
the \emph{standard G\"odel algebra}, denoted $T_{\min}$, and the \emph{standard Product algebra}, denoted $T_\times$, are defined analogously using (\ref{e:gtnorm}--\ref{e:prodtnorm}) and their residua. The important \emph{completeness theorems} for \LL, \G, and \PP\ will be tacitly assumed throughout: they state that these logics are complete with respect to evaluations into $T_{\odot}$, $T_{\min}$, and $T_{\times}$, respectively. For proofs and references, consult \cite[Theorems 3.2.13, 4.2.17, and 4.1.13, and \textit{passim}]{hajek}.

\smallskip In the remainder of this section we collect  technical results needed in the sequel. We begin with a remark that will find frequent application.
\begin{remark}\label{r:autos}
For any real-valued logic $(\Log, T_*)$, let $T_{*}'$ be an algebra of truth values that is isomorphic to $T_{*}$. Then
the logic induced by $T_*'$ is again $\Log$.
This follows immediately from the fact that $\sigma(1) = 1$ and $\sigma^{-1}(1) = 1$ for any isomorphism $\sigma\colon T_{*} \to T_{*}'$.
The converse statement is false in general: it is well known that non-isomorphic t-norms  may induce the same real-valued logic.
However, the following hold.
\begin{enumerate}
\item The only t-norm inducing $\G$ is the minimum operator, for  it is the only idempotent t-norm. See  \cite[Theorem 2.1.16]{hajek}.
\item Each t-norm inducing $\LL$ is isomorphic to $T_\odot$. See \cite[Lemmata 2.1.22.(2) and 2.1.23]{hajek}.
\item Each t-norm inducing $\PP$ is isomorphic to $T_\times$. See \cite[Lemma 2.1.22.(1)]{hajek}. \qed
\end{enumerate}
\end{remark}
\begin{lemma}\label{l:iffsemantics}For any real-valued logic $(\Log,T_*)$, and for any formul\ae\ $\alpha,\beta \in \Form$, we have:
\begin{align*}
\provesL\, \alpha \leftrightarrow \beta \ \ \ \ \Longleftrightarrow \ \ \ \  \ \provesL\, \alpha \to \beta\, \text{ and } \, \provesL\, \beta \to \alpha \ \ \ \ \Longleftrightarrow \ \ \ \ \mu(\alpha)=\mu(\beta) \text{ for all valuations } \mu \colon \Form \to T_*.
\end{align*}
\end{lemma}
\begin{proof}
Indeed, $\provesL\, \alpha \leftrightarrow \beta$ iff, by the completeness of $\Log$ with respect to $T_*$, for all valuations $\mu \colon \Form \to T_*$ we have $\mu(\alpha \leftrightarrow \beta) = 1$
iff, since $1$ is the neutral element for $*$,  $\mu(\alpha \to \beta) = \mu(\beta \to \alpha) = 1$ iff, by the completeness of $\Log$ with respect to $T_*$, $\provesL\, \alpha \to \beta\, \text{ and } \, \provesL\, \beta\to \alpha$ iff, since $\mu(\alpha\to\beta)=1$ is equivalent to $\mu(\alpha)\leq \mu(\beta)$ by the definition of residuum, $\mu(\alpha) = \mu(\beta)$.
\end{proof}
 BL-algebras are defined over the signature $( *, \to, \bot )$.  {\em Basic hoops} are the $\bot$-free subreducts of BL-algebras, the latter considered over the extended signature that includes $\top:=\bot\to\bot$. Conversely, BL-algebras are {\em bounded} basic hoops, that is, basic hoops with a minimum element which interprets the new constant $\bot$. Let now $(I, \leq)$ be a totally ordered set, and let $\{C_i\}_{i \in I}$ be a family of totally ordered basic hoops,
where $C_i := (C_i, *_i, \to_i, 1)$. Assume further that $C_i \cap C_j = \{1\}$ for each $i \neq j \in I$.
Then the {\em ordinal sum} of the family  $\{C_i\}_{i \in I}$ is the structure\footnote{Usage of the symbol $\oplus$ to denote ordinal sums seems fairly standard. It is also standard to use $\oplus$ to denote {\L}ukasiewicz's strong disjunction, see \cite{cdm}. This we will do in Section \ref{s:leibniz}, where context should prevent confusion.}
$$\bigoplus_{i \in I} C_i := \left(\, \bigcup_{i \in I} C_i,\, *\,,\to,\, 1\, \right)\,,$$
where
$$
x * y =
\left\{
\begin{array}{ll}
x *_i y & \mbox{ if } x,y \in C_i, \\
y & \mbox{ if } x \in C_i,\, y \in  C_j \setminus \{1\},\, i > j, \\
x & \mbox{ otherwise, }
\end{array}
\right.
$$
and
$$
x \to y =
\left\{
\begin{array}{ll}
x \to_i y & \mbox{ if } x,y \in C_i, \\
y & \mbox{ if } x \in C_i,\, y \in  C_j,\, i > j, \\
1 & \mbox{ otherwise. }
\end{array}
\right.
$$
Each $C_i$ is called a {\em summand} of the ordinal sum.

\smallskip
 \begin{lemma}[The Mostert-Shields Structure Theorem]\label{l:ordinalsumdecomposition} Each algebra of truth values $([0,1],*,\to_{*},0)$
 is isomorphic to an ordinal sum of bounded basic hoops, each of which
is isomorphic to one among $T_{\odot}$, $T_{\min}$, $T_{\times}$, and $\{0,1\}$.
\end{lemma}
\begin{proof}
This
is essentially \cite[Theorem B]{mostert_shields}.
\end{proof}

\begin{lemma}\label{l:subalgebraofordinalsum}
Let $A$ be a subalgebra of an ordinal sum $\bigoplus_{i \in I} B_i$.
Then there exists $J \subseteq I$ and algebras $\{C_j \mid j \in J\}$
such that $C_j$ is a subalgebra of $B_j$ for each $j \in J$, and
$A \cong \bigoplus_{j \in J} C_j$.
\end{lemma}
\begin{proof}
Direct inspection of the definition of ordinal sum.
\end{proof}
\emph{MV-algebras} \cite{cdm} are (term equivalent to) BL-algebras satisfying the equation $\neg\neg x =x$, where $\neg x$ is short for $x \to \bot$. \emph{Wajsberg hoops} are the $\bot$-free subreducts of MV-algebras; equivalently,
MV-algebras are exactly the bounded Wajsberg hoops.
\begin{lemma}\label{l:finiteordsum}
Each finite BL-chain splits into an ordinal sum of finitely many finite MV-chains.
\end{lemma}
\begin{proof}
This is \cite[Theorem 3.7]{agliano_montagna}, together with the observation that finite Wajsberg hoops are necessarily bounded.
\end{proof}
\begin{lemma}\label{l:2summands}Suppose the algebra of truth values $T_*$ is not a subalgebra of $T_{\odot}$. Then $T_*$ splits into a non-trivial ordinal sum of at least two
summands.
\end{lemma}
\begin{proof}
If $T_*$ is finite, from Lemma \ref{l:finiteordsum} it follows that $T_*$ is isomorphic to an ordinal sum of finitely many finite MV-chains. Since, by assumption, $T_*$ is not a subalgebra of $T_{\odot}$,
the ordinal sum must contain at least two summands.

If $T_*$ is an infinite subalgebra of $[0,1]$, by Lemmata \ref{l:ordinalsumdecomposition} and
\ref{l:subalgebraofordinalsum}
it follows that  $T_*$ is isomorphic to an ordinal sum $\bigoplus_{i \in I} C_i $ where each summand $C_i$ is isomorphic to
 a subalgebra of $T_\odot$, $T_{\min}$, or $T_\times$. If the index set $I$ has at least two elements, we are done; otherwise, by the hypotheses $T_*$ is  isomorphic
to a subalgebra of $T_{\min}$ or of $T_\times$, and it has more than two elements.
Now, by direct inspection, $T_{\min}$ is isomorphic to $\bigoplus_{r \in [0,1)} \{0,1\}$, while $T_\times$ is isomorphic to $\{0,1\} \oplus \mathcal{C}$, where $\mathcal{C} = (\, (0,1], \times, \to_\times, 1 \,)$
is known as the {\em standard cancellative hoop}.
Any subalgebra of $T_{\min}$ with more than two elements is then a non-trivial ordinal sum of copies of $\{0,1\}$,
while any subalgebra of $T_\times$ distinct from $\{0,1\}$ is of the form $\{0,1\} \oplus \mathcal{C}'$, for $\mathcal{C}'$ a subhoop of $\mathcal{C}$.
In both cases, $T_*$ splits into a non-trivial ordinal sum of at least two
summands.
\end{proof}
\begin{lemma}\label{l:varseval}Suppose the algebra of truth values $T_*$ splits into a non-trivial ordinal sum of at least two summands, say $\bigoplus_{i \in I} C_i$,
where each $C_{i}$ is a totally ordered basic hoop, and $|I|\geq 2$. Then $I$ has a least element, say $i_{0}$. Further,  let $S\subseteq T_*$ be the support of
a summand distinct from $C_{i_{0}}$.
For any two valuations $\mu, \nu \colon \Form \to T_*$ such that $\mu(\Var), \nu{(\Var)}\subseteq S$, and for any $\alpha \in\Form$, we have:
\begin{align*}
\mu(\alpha)=0 \ \ \ \ \Longleftrightarrow \ \ \ \ \nu(\alpha)=0.
\end{align*}
\end{lemma}
\begin{proof}Since $T_{*}$ is bounded below, the existence of $i_{0}$ follows from inspection of the definition of ordinal sum.

We first  prove the following \emph{claim} by induction on the structure of formul\ae:
For any valuation $\mu \colon \Form \to T_*$ such that $\mu{(\Var)}\subseteq S$, and for any $\alpha \in\Form$,
we have $\mu(\alpha) \in S \cup \{0\}$.

If $\alpha$ is either $\bot$ or $\alpha \in \Var$, the claim holds trivially.
Suppose $\alpha=\beta \& \gamma$.
By the induction hypothesis,  $\mu(\beta),\mu(\gamma) \in S \cup \{0\}$.
If both  $\mu(\beta),\mu(\gamma)\in S$ then, by the definition of ordinal sum, $\mu(\beta \& \gamma) \in S$, too.
If at least one among $\beta$ and $\gamma$, say $\beta$, is such that $\mu(\beta) = 0$, then $\mu(\beta \& \gamma) = 0$.
Hence $\mu(\beta \& \gamma) \in S \cup \{0\}$ for all $\mu$ such that $\mu{(\Var)}\subseteq S$.
Next suppose $\alpha=\beta \to \gamma$.
If $\mu(\beta) \leq \mu(\gamma)$, then $\mu(\beta \to \gamma) = 1 \in S$.
If $\mu(\beta) > \mu(\gamma) \in S$ then, by the definition of ordinal sum, $\mu(\beta \to \gamma) \in S$, too.
Finally, if  $\mu(\beta) \in S$ and $\mu(\gamma) = 0$, then $\mu(\beta \to \gamma) = 0$. In all cases $\mu(\beta \to \gamma) \in S \cup \{0\}$.
This settles the \emph{claim}.

Consider now $\mu, \nu \colon \Form \to T_*$ such that $\mu(\Var), \nu{(\Var)}\subseteq S$, and any formula $\alpha \in \Form$.
It suffices to show that $\mu(\alpha) = 0$ implies $\nu(\alpha)= 0$.
By the preceding claim, we have $\mu(\alpha),\nu(\alpha) \in S \cup \{0\}$.
We proceed again by induction on the structure of formul\ae.
The base cases $\alpha=\bot$ or $\alpha \in \Var$ hold trivially.
Let $\alpha=\beta \& \gamma$. The definition of ordinal sum entails that $\mu(\beta \& \gamma) = 0$ can only occur if
at least one of
$\mu(\beta)$ and $\mu(\gamma)$, say $\mu(\beta)$, lies in the first summand $C_{i_{0}}$. By the preceding claim,
$\mu(\beta) = 0$. By induction
$\nu(\beta) = 0$, and therefore $\nu(\beta \& \gamma) = 0$.
Let $\alpha=\beta \to \gamma$.
Assume $\mu(\beta \to \gamma) = 0$. The definition of ordinal sum entails
either
$\mu(\beta) > \mu(\gamma) = 0$,
or both $\mu(\beta),\mu(\gamma) \in C_{i_0}$.
In the latter case, the preceding claim shows $\mu(\beta) = \mu(\gamma) = 0$, and therefore
$\mu(\beta \to \gamma) =1$, which is a contradiction.
In the former case,  by induction
$\nu(\beta)  > \nu(\gamma) = 0$. By the preceding claim, $\nu(\beta) \in S$.
By the definition of ordinal sum
$\nu(\beta \to \gamma) = 0$.
This completes the proof.
\end{proof}
\section{Logics satisfying \P{1}.}\label{s:tarski}
\begin{lemma}\label{l:mintnorm}For any real-valued logic $(\Log,T_*)$, we have:
\begin{align*}
\Log \text{ extends } \G \ \ \ \ \Longleftrightarrow \ \ \ \ T_* \text{ is a subalgebra of } T_{\rm min}.
\end{align*}
Moreover, we have:
\begin{align*}
\Log \text{ extends } \G \text{ properly \textup{(}\ie\ $\Log\neq \G$\textup{)}} \ \ \ \ \Longleftrightarrow \ \ \ \ T_* \text{ is a finite subalgebra of } T_{\rm min}.
\end{align*}

\end{lemma}
\begin{proof}
$\Log$ extends $\G$ iff $\provesL\, X_1 \leftrightarrow X_1\&X_1$ iff,
by Lemma \ref{l:iffsemantics},
$\mu(X_1) = \mu(X_1) * \mu(X_1)$ for any valuation
$\mu \colon \Var \to T_*$ iff $a = a * a$ for any $a \in T_*$ iff $T_*$ is a subalgebra of $T_{\min}$. (The latter equivalence follows from   Remark \ref{r:autos}.1.)
Now, if $\Log$ extends $\G$ properly, then, by Remark \ref{r:autos}.1,
and the  fact that each infinite subalgebra of $T_{\rm min}$ induces $\G$ \cite[Theorem 4]{dummett59},
the underlying set of $T_*$ cannot be
an infinite subset of $[0,1]$, hence $T_*$ is a finite subalgebra of $T_{\min}$.
The other direction follows from \cite[Corollary 4.2.15]{hajek}, stating that any two finite subalgebras of $T_{\min}$ of the same cardinality are isomorphic,
and from the axiomatisation of the subvariety of G\"{o}del algebras generated by the $n$-element chain, essentially given in \cite{godel}.
\end{proof}
\begin{lemma}\label{l:maintarski}Any  real-valued logic that satisfies \P{1} is an extension of \G.
\end{lemma}
\begin{proof}We prove the contrapositive: a real-valued logic $\Log$ that does not extend  \G\ fails \P{1}. Indeed, by the hypothesis we have $\nprovesL\, X_1\leftrightarrow X_1\&X_1$. On the other hand, for any algebra of truth values $T_{*}$ inducing $\Log$, and for any valuation $\mu\colon \Form\to T_*$, we have
\begin{align}
\mu(X_1)=1 \ \ &\Rightarrow \ \ \mu(X_1\&X_1)=1,\label{d:firstimpl}\\
\mu(X_1\&X_1)=1 \ \ &\Rightarrow \ \ \mu(X_1)=1.\label{d:secondimpl}
\end{align}
Indeed, (\ref{d:firstimpl}) holds by the very definition of t-norm, which includes the condition $1*1=1$; and (\ref{d:secondimpl}) holds by the fact that t-norms are non-increasing in both arguments, whence $\mu(X_1\&X_1)\leq \mu(X_1)$. Now (\ref{d:firstimpl}--\ref{d:secondimpl}) show that $\Log$ fails \P{1} for $\alpha=X_1$ and $\beta =X_1\&X_1$.
\end{proof}

For the proof of the next lemma we recall the notion of  semantic consequence with respect to an algebra of truth values $T_*$.
Given a set $\Gamma \subseteq \Form$ and $\alpha \in \Form$, we say that $\alpha$ is a \emph{semantic consequence}
of $\Gamma$ with respect to $T_*$, in symbols $\Gamma \vDash_{T_*} \alpha$ if, for any valuation $\mu \colon \Var \to T_*$,
the fact that $\mu(\gamma) = 1$ for each $\gamma \in \Gamma$ implies $\mu(\alpha) = 1$.

\begin{lemma}\label{l:godimpliestarski}Any real-valued logic $\Log$  that  is an extension of $\G$ satisfies $\P{1}$.
\end{lemma}
\begin{proof} Let $T_*$ be an algebra of truth values inducing $\Log$.
 By Lemma \ref{l:mintnorm} we know that $T_*$ is a subalgebra of $T_{\rm min}$. Let $\alpha, \beta \in\Form$ be such that $\mu(\alpha)=1$ iff $\mu(\beta)=1$, for each valuation $\mu\colon \Form \to T_*$. By the definition of semantic consequence,
we have $\alpha \vDash_{T_*} \beta$ and $\beta \vDash_{T_*} \alpha$.
Recall that $\G$ is strongly complete with respect to $T_{\min}$ (\cite[Theorem  4.2.17.(2)]{hajek}).
 By Lemma \ref{l:mintnorm}, each real-valued extension $\Log$ of $\G$  distinct from $\G$ is induced by a finite subalgebra of $T_{\min}$,   and it is moreover strongly complete with respect to any such (essentially unique) subalgebra (\cite[Proposition 4.18 and Corollary 4.19]{Distinguished}).
In all cases we therefore infer $\alpha \, \provesL \, \beta$ and $\beta \, \provesL \, \alpha$. The logic \G\ has the Deduction Theorem by \cite[Theorem 4.2.10.(1)]{hajek}, and the same proof shows that each extension of \G\ also has the Deduction Theorem. We thereby obtain
$\provesL\, \beta \to \alpha$ and $\provesL\,\alpha \to \beta$. Hence, by Lemma \ref{l:iffsemantics}, we conclude $\provesL\, \alpha\leftrightarrow\beta$, as was to be shown.
\end{proof}

\paragraph{Proof of Theorem I} Combine Lemmata \ref{l:maintarski} and \ref{l:godimpliestarski}. \qed

\begin{remark}\label{r:godelandMTL}
Theorem I holds even if we   relax the notion of real-valued logic considerably.
Recall that \emph{MTL} (\emph{monoidal t-norm-based logic}) is the logic of all left-continuous t-norms and their residua \cite{eg};  write  $\Form'$ for the set of well-formed formul\ae\ of MTL. (In contrast to BL,  here  it is necessary  to regard the lattice-theoretic conjunction $\wedge$ as primitive.) The algebraic semantics corresponding to MTL is provided by \emph{MTL-algebras}.
By a \emph{standard MTL-algebra} we mean an MTL-algebra induced by a left-continuous t-norm on $[0,1]$ and its residuum.
Now  replace the definition of real-valued logic by the following. The pair $(\Log, T_*)$
is a {\em real-valued logic} if $\Log$ is an extension of MTL that is complete with respect
to valuations $\mu \colon \Form' \to T_*$ into an arbitrary MTL-subalgebra $T_{*}$
of some standard MTL-algebra.
It is well known that Remark \ref{r:autos}.1 holds even
if we consider all left-continuous t-norms instead of the continuous ones only.
And it is possible to show that Lemmata \ref{l:mintnorm}, \ref{l:maintarski}, and \ref{l:godimpliestarski}
continue to hold. Hence Theorem I holds for real-valued logics in the present sense.
\qed
\end{remark}
\section{Logics satisfying \P{2}.}\label{s:leibniz}
\begin{lemma}\label{l:extL}For any real-valued logic $(\Log,T_*)$, we have:
\begin{align*}
\Log \text{ extends } \LL \ \ \ \ \Longleftrightarrow \ \ \ \ T_* \text{ is isomorphic to a  subalgebra of } T_{\odot}.
\end{align*}
Moreover, we have:
\begin{align*}
\Log \text{ extends } \LL \text{ properly \textup{(}\ie\ $\Log\neq \LL$\textup{)}}\ \ \ \ \Longleftrightarrow \ \ \ \ T_* \text{ is isomorphic to a finite subalgebra of } T_{\odot}.
\end{align*}
\end{lemma}
\begin{proof}

$\Log \text{ extends } \LL$ iff $\vdash_\Log \neg\neg X_1 \leftrightarrow X_1$ iff (by Lemma \ref{l:iffsemantics}) $\mu(X_1) = \neg\neg\mu(X_1)$ for any valuation $\mu \colon \Var \to T_*$
iff $a = \neg\neg a$ for any $a \in T_*$ iff (by Remark \ref{r:autos}.2) $T_*$ is an MV-algebra with some underlying set $U \subseteq [0,1]$. Now, if $U$ is finite, say of cardinality $n$,
then $T_*$ is isomorphic to the MV-chain $T_{n-1} = \{\frac{0}{n-1},\frac{1}{n-1},\ldots,\frac{n-2}{n-1},\frac{n-1}{n-1}\}$, by \cite[Proposition 3.6.5]{cdm},
and direct inspection shows that $T_{n-1}$ is a subalgebra of $T_\odot$.
Assume then that $U$ is infinite. Observe that $T_*$ cannot be a non-trivial ordinal sum of at least two summands: consider such a sum $B \oplus C$, and take $1 \neq c \in C$.
Then $\neg\neg c = 1 \neq c$, and hence $B \oplus C$ is not an MV-algebra.
By Lemma \ref{l:2summands} and  by Remark \ref{r:autos}.2,
$T_*$ is isomorphic to
a subalgebra of
$T_\odot$.
Clearly, if $T_*$ is isomorphic to a subalgebra of
$T_\odot$ then $\Log \text{ extends } \LL$. This proves the first statement.
Each finite MV-chain  generates a proper subvariety of the variety of MV-algebras (see \cite[Theorem 8.5.1]{cdm} for axiomatisations). Thus,
if $T_*$ is isomorphic to a finite subalgebra of $T_\odot$ then $\Log \text{ extends } \LL$ properly.
On the other hand, by \cite[Theorem 8.1.1]{cdm}, every infinite subalgebra of $T_\odot$ generates the whole variety of MV-algebras. This fact, together
with the first assertion of the lemma, suffices to complete the proof.
\end{proof}
\begin{lemma}\label{l:p2impliesL} Any real-valued logic $\Log$ that satisfies \P{2} is an extension of $\LL$.
\end{lemma}
\begin{proof}By contraposition, suppose $\Log$ is not an extension of $\LL$. If $T_{*}$ is an algebra of truth values that induces $\Log$, then $T_{*}$ is not a subalgebra of $T_{\odot}$: for, given that $T_{\odot}$ does induce $\LL$ (\cf\ Remark \ref{r:autos}), any such subalgebra clearly induces an extension of $\LL$.
Hence, by Lemma \ref{l:2summands}, $T_{*}$ splits into a non-trivial ordinal sum of at least two summands.  With the notation therein, there exists a summand $S$ of $T_{*}$ distinct from the first one that is non-trivial, and thus contains  two distinct elements $v\neq w$. Let
$\mu_{v}$ be the unique valuation that sends each variable to $v$, and let $\nu_{w}$ be the unique valuation that sends each variable to $w$. Evidently, we have $\mu_{v}\neq \nu_{w}$, so that  $\mu_{v}$ and $\nu_{w}$ fail \P{2} by Lemma \ref{l:varseval}.
\end{proof}
\begin{remark}Let $(\Log, T_*)$ be a real-valued logic. In the next lemma we say, somewhat informally, that ``$\Log$ satisfies \P{2} with respect to $T_*$'', to mean that for any two valuations $\mu\neq \nu \colon\Form \to T_*$ there is   $\alpha\in\Form$ with $\mu(\alpha)>0$ and $\nu(\alpha)=0$. \qed
\end{remark}
\begin{lemma}\label{l:inv}Let $(\Log, T_{*})$ be a real-valued logic, and let $\sigma \colon T_{*}\to T'_{*'}$ be an isomorphism, where $T'_{*'}$ is an algebra of truth values. The logic induced by $T'_{*'}$ is again $\Log$, by Remark \ref{r:autos}. Then
$\Log$ satisfies \P{2} with respect to $T_{*}$ if, and only if, $\Log$ satisfies \P{2} with respect to $T'_{*'}$.
\end{lemma}
\begin{proof}Since $\sigma^{-1}\colon T'_{*'}\to T_{*}$ is an isomorphism, too, it suffices to show that $\Log$  satisfies \P{2} with respect to $T'_{*'}$ if
$\Log$ satisfies \P{2} with respect to $T_{*}$. Proof by contraposition. Let $\mu\neq\nu\colon \Form\to T'_{*'}$ be valuations that
fail \P{2}. Thus, for all formul\ae\ $\alpha \in \Form$, we have $\mu(\alpha)=0$ if, and only if, $\nu(\alpha)=0$. Write $\mathrm{Free}_{\aleph_{0}}^{\Log}$ for the Lindenbam algebra of the logic $\Log$. As usual, we may identify  formul\ae, modulo the logical-equivalence relation induced by $\provesL$, with elements of $\mathrm{Free}_{\aleph_{0}}^{\Log}$; and valuations with homomorphisms from $\mathrm{Free}_{\aleph_{0}}^{\Log}$ to $T_{*}$ (or to $T'_{*'}$, as the case may be). Then the compositions $\sigma^{-1}\circ \mu$ and $\sigma^{-1}\circ\nu$ are valuations into $T_{*}$, see the commutative diagram below.
\[
\begin{tikzpicture}[scale=0.3]
\node(A) at (0,5)   {${\mathrm{Free}_{\aleph_{0}}^{\Log}}$};
\node (P1) at (10,5) {$T_{*}$};
\node (P2) at (10,-4) {$T'_{*'}$};
\draw[transform canvas={yshift=0.4ex},<-] (P1) -- (A) node [above, midway] {$\sigma^{-1}\circ\mu$};
\draw[transform canvas={yshift=-0.4ex},<-] (P1) -- (A) node [below, midway] {$\sigma^{-1}\circ\nu$};
\draw[transform canvas={xshift=0.4ex},<-] (P2) -- (A) node [above right, midway] {$\mu$};
\draw[transform canvas={xshift=-0.5ex, yshift=-0.3ex},<-] (P2) -- (A) node [below left, midway] {$\nu$};
\draw [->] (P2) -- (P1) node [right, midway] {$\sigma^{-1}$};
\end{tikzpicture}
\]
It is not the case that $\sigma^{-1}\circ \mu = \sigma^{-1}\circ \nu$: for else $\mu=\nu$ would follow  by pre-composing with $\sigma$. Now for any $\alpha\in \Form$ we have:
\begin{align*}
\mu(\alpha)=0 \ \ \ & \text{iff} \ \ \ \ \nu(\alpha)=0 & \text{(by assumption),}\\
\sigma^{-1}(0)=0 \ \ \ &  & \text{(homomorphisms preserve $0$),}\\
\sigma^{-1}(\mu(\alpha)))=0 \ \ \ & \text{iff} \ \ \ \ \sigma^{-1}(\nu(\alpha)))=0 & \text{(by composition).}
\end{align*}
Hence $\Log$ fails \P{2} with respect to $T_{*}$, as was to be shown.
\end{proof}
\begin{lemma}\label{l:separation}{\L}ukasiewicz logic \LL\   satisfies \P{2}.
\end{lemma}
\begin{remark}\label{rem:man}A proof of Lemma \ref{l:separation} can be obtained  as a consequence of McNaughton's Theorem \cite[9.1]{cdm}; in fact, the proof can be
 reduced to the one-variable case \cite[3.2]{cdm}. Here we give a proof that uses a weaker (and simpler) result from \cite{agu}, thus showing that the full strength of McNaughton's Theorem is not needed to fulfill \P{2}.\qed
\end{remark}
\begin{proof}
In light of Remark \ref{r:autos}.2 and Lemma \ref{l:inv},
it suffices to show that \LL\ satisfies \P{2} with respect to the  {\L}ukasiewicz t-norm $\odot$ on  $[0,1]$. For terms $s$ and $t$ over the binary monoidal operation $\odot$ and the unary operation $\neg$, set $s\oplus t :=\neg(\neg s\odot \neg t)$. Let us write $nt$ as a shorthand for $t\oplus\cdots \oplus t$ ($n-1$ occurrences of $\oplus$), and
$t^{n}$ as a shorthand for $t\odot\cdots \odot t$ ($n-1$ occurrences of $\odot$). We inductively define the set of \emph{basic literals} (in the variables $X_{i}$, $i=1,2,\ldots$) as follows.
\begin{itemize}
\item $X_{i}$ is a basic literal;
\item each term $s$ either of the form $s=nt$ or of the form $s=t^{n}$, for  some  integer $n>0$, is a basic literal, provided that  $t$ is a basic
literal;
\item nothing else is a basic literal.
\end{itemize}
Given integers $n_{1}\geq 1$, and $n_{2},\ldots,n_{u}>1$, we write $(n_{1},n_{2},\ldots,n_{u})X_{i}$ to denote the basic literal $$(\cdots ((n_{i}\cdots ((n_1X_{i})^{n_2} \cdots ))^{n_{i+1}})\cdots ).$$
In this proof, a \emph{term function} is any function $\lambda_{\tau} \colon [0,1]^{n}\to [0,1]$ induced by interpreting over the standard MV-algebra $T_{\odot}=([0,1],\odot,\neg,0)$ a term $\tau$ whose variables are contained in $\{X_1,\ldots,X_{n}\}$. Below we also use the  interpretation of the definable lattice connective $\wedge$ as the minimum operator.
\begin{claim}\label{c:separation}
For any integer $n \geq 1$, and for any two points $p\neq q\in[0,1]^{n}$, there is a term $\tau$ whose term function $\lambda_{\tau} \colon [0,1]^{n}\to [0,1]$  takes value $0$ at $q$, and value $>0$ at $p$.
\end{claim}
\begin{proof}
Since $p \neq q$ there exists an integer $i\geq 1$ such that $p(i) \neq q(i)$, that is, $p$ and $q$ differ at one of their coordinates.
If $q(i) < p(i)$ then there are integers  $h,k>0$ such that $q(i) < \frac{h}{k} < p(i)$, with $h$ and $k$ coprime.
By \cite[Corollary 2.8]{agu} there is a basic literal $L =(a_1,\ldots, a_{u})X_{i}$ such that
$\lambda_{L}^{-1}(0)$ is the set $[0,\frac{h}{k}] \times [0,1]^{n-1}$, and $\lambda_{L}$ is monotone increasing in the variable $X_i$.
Hence $\lambda_L(p) > 0$ and $\lambda_L(q) = 0$.
If $p(i) \leq q(i)$ for all integers $i \geq 1$, then one can choose $j$ such that $p(j) < q(j)$.
As before there are integers  $h,k>0$ such that $p(j) < \frac{h}{k} < q(j)$, with $h$ and $k$ coprime,
and there is a basic literal $R = (b_1,\ldots, b_{w})X_{i}$ such that $\lambda_{R}^{-1}(1)$ is the set $[\frac{h}{k},1] \times [0,1]^{n-1}$, and $\lambda_{R}$ is monotone increasing
in the variable $X_i$.
Hence $\lambda_{\neg R}(p) > 0$ and $\lambda_{\neg R}(q) = 0$.
\end{proof}
The proof is now completed by a routine translation of Claim \ref{c:separation} from terms to formul\ae\ of $\LL$.
\end{proof}
\begin{remark}In connection with Claim \ref{c:separation}, let us observe that term functions in {\L}ukasiewicz logic (even over an arbitrarily large set $I$ of propositional variables) enjoy an even stronger separation property. Recall (see \eg\ \cite[1.5]{engelking}) that a space is \emph{completely regular} if it is $T_{1}$, and points can be separated from closed sets by continuous $[0,1]$-valued functions. Now, \emph{in each product space $[0,1]^{I}$,  points can be separated from closed sets by term functions}.  Thus  the space of standard models $[0,1]^{I}$ may be described as \emph{definably completely regular}. The proof is essentially the same as the one above, \textit{mutatis mutandis}; \cf\ \cite[Lemma 3.5]{marraspada}.
\qed
\end{remark}

\paragraph{Proof of Theorem II}In light of Lemmata \ref{l:p2impliesL} and \ref{l:separation}, it remains to show that each real-valued extension of $\LL$ that is not $\LL$ itself satisfies \P{2}. By Lemmata \ref{l:extL} and \ref{l:inv}, we may safely assume that $\Log$ is induced by  a finite subalgebra $T_{*}$ of $T_{\odot}$.
By \cite[Proposition 3.6.5]{cdm}, each such subalgebra is isomorphic to $T_m = \left\{\frac{0}{m},\frac{1}{m},\ldots,\frac{m-1}{m},\frac{m}{m}\right\}$, for a uniquely determined integer $m\geq 1$.
Notice now that if $p\neq q$ are in $T_m^{n}$ then the term function $\lambda'_{\tau}$ obtained by restricting to $T_m^{n}$ the function $\lambda_{\tau} \colon [0,1]^{n}\to [0,1]$
provided by Claim \ref{c:separation} is such that $\lambda'_{\tau}(q) = 0$ while $\lambda'_{\tau}(p) > 0$.
Hence $\Log$ satisfies \P{2}, and the proof is complete. \qed\\

To conclude this section, let us discuss two alternative formulations of \P{2}.  We consider the following conditions, for every algebra  $T_*$ of truth values inducing $\Log$.
\princ{\P{2$'$}}{For each pair of valuations $\mu, \nu \colon \Form \to T_{*}$, if $\mu\neq \nu$ then
there is a formula $\alpha\in \Form$ such that $\mu(\alpha) <1$ while $\nu(\alpha) = 1$.\qed\\}
\princ{\P{2$''$}}{For each pair of valuations $\mu, \nu \colon \Form \to T_{*}$, if $\mu\neq \nu$ then
there is a formula $\alpha\in \Form$ such that $\mu(\alpha) =0$ while $\nu(\alpha) = 1$.\qed\\}
\begin{corollary}\label{cor:equiv}A real-valued logic satisfies \P{2} if, and only if, it satisfies \P{2$'$} if, and only if, it satisfies \P{2$''$}.
\end{corollary}
\begin{proof}
Let $T_*$ be an algebra of truth-values inducing the real-valued logic $\Log$.
It suffices to prove that if $\Log$ is an extension of $\LL$ then it satisfies $\P{2}'$ and $\P{2}''$,
and otherwise it fails both.

Assume first that $\Log$ is an extension of $\LL$.
Given valuations $\mu \neq \nu$ with values in $T_{*}$, by Theorem II there is a formula $\alpha$ be such that $\mu(\alpha) > 0$ and $\nu(\alpha) = 0$.
Then $\mu(\neg\alpha) < 1$ and $\nu(\neg\alpha) = 1$. Hence \P{2$'$} holds. We now show that \P{2$'$} implies \P{2$''$}.
In light of Remark \ref{r:autos}.2 and Lemma \ref{l:inv}, we may safely assume that $T_*$ is a subalgebra of $T_\odot$.
Then,
if $\mu(\alpha) < 1$ and $\nu(\alpha) = 1$, it is clear by the definition of $\odot$ that there exists an integer $k \geq 1$ such that
$\mu(\alpha^k) = 0$ and $\nu(\alpha^k) = 1$, where $\alpha^1 = \alpha$ and $\alpha^n = \alpha \odot \alpha^{n-1}$.

Assume now $T_*$ does not induce an extension of $\LL$.
By Theorem II,
there are distinct valuations $\mu$ and $\nu$ such that
$\nu(\alpha) = 0$ implies $\mu(\alpha) = 0$ for
any formula $\alpha$.
This suffices to show that \P{2}$''$ fails. 
For what concerns \P{2}$'$,
recall that, by Lemma \ref{l:extL} and Lemma \ref{l:2summands}, $T_*$ splits into a non-trivial ordinal sum of at least two summands.
Let $\mu$ be the valuation assigning $1$ to every variable.
Then it is easy to check that $\mu(\alpha) \in \{0,1\}$ for each formula $\alpha$.
Let $\nu$ be a valuation such that $\nu(\Var)$ is contained in a summand of $T_*$ distinct from the first one.
Then, by Lemma \ref{l:varseval}, for each formula $\alpha$ we have
$\nu(\alpha) = 0$ iff $\mu(\alpha) = 0$, and hence $\nu(\alpha) =1$ implies $\mu(\alpha) = 1$,
that is, \P{2}$'$ fails.
\end{proof}

\section{Product logic.}\label{s:product}

\begin{lemma}\label{l:noext}The only many-valued logic that extends \PP\ properly is classical logic.
\end{lemma}
\begin{proof}

This is essentially \cite[Corollary 2.10]{cignoli_torrens}.
\end{proof}
\begin{lemma}\label{l:fails}Product logic \PP\ fails both \P{1} and \P{2}.
\end{lemma}
\begin{proof} (\P{1}) \, Choose the standard product algebra $T_\times$ to induce \PP.  It follows directly from the definition of  t-norm that $\mu(X_1)=1$ if, and only if, $\mu(X_1\&X_1)=1$, for any valuation $\mu\colon\Form\to T_\times$. To see that \P{1} fails,  it thus suffices to observe that $\not \vdash_{\PP} X_1 \leftrightarrow X_1 \& X_1$: for else, by soundness and Lemma \ref{l:iffsemantics}, we would have $\mu(X_1\&X_1)=\mu(X_1)\mu(X_1)=\mu(X_1)$ whatever $\mu$ is; this is a contradiction.

\smallskip \noindent (\P{2}) \,
By Remark \ref{r:autos}.3 and Lemma \ref{l:inv}, it suffices to argue about the product t-norm $T_\times$. By direct inspection, we have the decomposition $T_\times = \{0,1\} \oplus \mathcal{C}$, where $\mathcal{C}$ is the standard cancellative hoop. The hypotheses of Lemma \ref{l:varseval} are therefore satisfied, and hence
 \P{2} fails for any two valuations $\mu\neq \nu \colon \Form \to T_\times$ such that $\mu(\Var),\nu(\Var)\subseteq \mathcal{C}$.
\end{proof}
\begin{lemma}\label{l:justproduct}Let $\Log$ be a closed real-valued logic all of whose non-classical, real-valued extensions fail \P{1} and \P{2}. Then $\Log=\PP$.
 \end{lemma}
\begin{proof} We know that $\Log$ is not an extension of $\G$ or  $\LL$, by Theorems I and II. Let $T_*$ be any algebra of truth values inducing $\Log$. We will show that $T_*$ cannnot be finite, to begin with.

If $T_*$ is finite,
by Lemma \ref{l:finiteordsum} we know that $T_*$ splits into an ordinal sum of finitely many finite MV-chains. If there is just one summand, then $\Log$ is an extension of $\LL$, and this is a contradiction. If there is more than one summand then, by the  definition of ordinal sum, and using the fact that each summand is bounded below by $0$,  there is
an idempotent element $0,1\neq e \in T_*$. The subset $G_3:=\{0,e,1\}\subseteq T_*$ is closed under the BL-algebraic operations, as is checked easily, and all of its elements
are idempotent. Hence $G_3$ is isomorphic to the three-element G\"odel algebra. Now consider the collection $\mathscr{E}$ of formul\ae\ that evaluate to $1$ under each valuation into $G_3$. Obviously $\mathscr{E}\supseteq \Log$, and $\mathscr{E}$ is closed under substitutions by its very definition.  Hence $\mathscr{E}$ is a real-valued extension of $\Log$ which by construction is three-valued G\"odel logic.  Theorem I implies that $\mathscr{E}$  satisfies \P{1}, and we have reached a contradiction.

We may therefore suppose that $T_*$ has an infinite closed subset of $[0,1]$ as its support. By definition, $T_{*}$ extends to a BL-algebra $([0,1],*',\to_{*'},0)$.
By Lemmata \ref{l:ordinalsumdecomposition} and \ref{l:subalgebraofordinalsum}, $T_*$ decomposes into an ordinal sum $\bigoplus_{i \in I} C_i$, where each summand $C_i$ is
isomorphic to  a subalgebra of one amongst $T_\odot$, $T_{\min}$, and  $T_\times$. If the index set $I$ has more than one element, then using again the fact that each summand $C_i$
is bounded below by $0$, we have an idempotent element $0,1\neq e \in T_*$, and hence $\{0,e,1\}$ is a three-element  G\"odel subalgebra of $T_*$.
We then reason as above to conclude that $\Log$ has three-valued G\"odel logic as an extension,  reaching a contradiction.
Hence $I$ is a singleton, that is, $T_*$ is isomorphic to a subalgebra of $T_\odot$, $T_{\min}$, and $T_\times$. Using Remark \ref{r:autos},
and Theorems I and II, $T_*$ cannot be isomorphic to a subalgebra of $T_\odot$ --- because it fails \P{2} --- nor can it be isomorphic to a subalgebra of $T_{\min}$
--- because it fails \P{1}. Then $T_*$ is isomorphic to an infinite subalgebra of $T_\times$, and hence $\Log=\PP$, by \cite[Corollary 2.9]{cignoli_torrens}.
\end{proof}
\paragraph{Proof of Theorem III}  Lemmata \ref{l:noext}, \ref{l:fails} and \ref{l:justproduct}. \qed
\begin{remark}\label{rem:fails}
Theorem III fails if we drop the assumption that $\Log$ be closed. Indeed, consider the logic $\Log$ induced by $\{0,1\} \oplus  \mathcal{C}\oplus \mathcal{C}$, where $\mathcal{C}$ is the standard cancellative hoop (see the proof of Lemma \ref{l:2summands}).
Then it can be verified that $\Log$ is not closed, that $\Log$ is not $\PP$,  and that all of its non-classical,
real-valued extensions fail \P{1} and \P{2}. \qed
\end{remark}

\section{Epilogue.} Let us return to H\'ajek's Programme, as  embodied in \cite{hajek}.  According to H\'ajek, a real-valued logic may be considered as a ``{\sl logic of imprecise \textup{(}vague\textup{)} propositions}'' \cite[p.vii]{hajek}, wherein  ``{\sl truth \textup{[}\ldots\textup{]} is a matter of degree}'' \cite[p.2]{hajek}. Classical logic may be viewed as a limiting case, where only two degrees of truth, $0$ and $1$, exist. But as soon as a logic is genuinely real-valued, it must renounce  at least one of the familiar features  \P{1} and \P{2} of the classical world. We record this fact as a formal statement, by way of conclusion.
\begin{corollarynonum}A real-valued logic $\Log$ satisfies $\P{1}$ and $\P{2}$ if, and only if, $\Log$ is classical logic if, and only if, $T_*=\{0,1\}$ is the unique  algebra of truth values that induces $\Log$.
\end{corollarynonum}
\begin{proof}That $\Log$ is classical logic just in case $\Log$ satisfies $\P{1}$ and $\P{2}$ follows from Theorems I--II upon observing that the only common  extension of $\G$ and $\LL$ is classical logic, by \cite[Theorem 4.3.9.(1)]{hajek}. It thus remains to show that $\Log$ is classical logic if, and only if, $T_*=\{0,1\}$ as soon as $T_*$ induces $\Log$.
By the very definition of t-norm,  $T_* = \{0,1\}$ induces classical logic. On the other hand,  if there exists $a \in T_* \setminus \{0,1\}$ then  $\max{\{a, a \to_{*} 0\}}< 1$. Indeed, $a \to_{*} 0 = 1$ would entail $a * 1 = 0$ for $a > 0$, which is impossible. Any valuation $\mu\colon \Form \to T_{*}$ that sends $X_{1}$ to $a$ is therefore such that $\mu(X_{1}\vee\neg X_{1})<1$, and the logic induced by $T_*$ cannot be classical.
\end{proof}

\section*{Acknowledgements.}
\noindent We are grateful to two anonymous referees for several remarks on an earlier version of this paper that led to improvements in exposition, and to shorter proofs of some of the results given here.

\bibliographystyle{plain}

\end{document}